\documentclass[final,leqno]{siamltex}

\usepackage{amsmath,amsfonts}
\usepackage{enumerate}
\usepackage{caption}
\usepackage{subcaption}
\usepackage{graphicx}
\usepackage{epstopdf}
\usepackage{bbm}
\usepackage{float}
\usepackage{verbatim}

\title{SHRINAKGE FUNCTION AND ITS APPLICATIONS IN MATRIX APPROXIMATION}

\author{Toby Boas\thanks{Department of Mathematics, University of Florida, Gainesville, FL 32611-8105 ({tboas@ufl.edu}).}
        \and Aritra Dutta \thanks{Department of Mathematics, University of Central Florida, 4000 Central Florida Blvd, Orlando, Florida 32816({d.aritra2010@knights.ucf.edu}).}
        \and Xin Li \thanks{Department of Mathematics, University of Central Florida, 4000 Central Florida Blvd, Orlando, Florida 32816 ({xli@ucf.edu}).}
        \and Kathryn P. Mercier\thanks{Department of Biology, University of Central Florida, 4000 Central Florida Blvd, Orlando, Florida 
        	32816({katie.mercier@knights.ucf.edu}).}
        \and Eric Niderman\thanks{Department of Mathematics, University of Central Florida, 4000 Central Florida Blvd, Orlando, Florida 
        32816 ({eaniederman@knights.ucf.edu}).}
        }

\begin{document}

\maketitle

\begin{abstract}The shrinkage function is widely used in matrix low-rank approximation, compressive sensing, and statistical estimation.~In this article, an elementary derivation of the shrinkage function is given.~In addition, applications of the shrinkage function are demonstrated in solving several well-known problems, together with a new result in matrix approximation. 
\end{abstract}

\begin{keywords}
	Shrinkage function, singular value decomposition, low-rank approximation, sparse approximation. 
\end{keywords}

\begin{AMS}
	65F15, 65F30, 65F35, 65F50, 65K10
\end{AMS}

\pagestyle{myheadings}
\thispagestyle{plain}
\markboth{BOAS, DUTTA, LI, MERCIER AND NIDERMAN}{Shrinkage function and its applications}

\section{Introduction} Historically some important mathematical functions have been introduced for convenience. For example, the Heaviside step function $H(\cdot)$, a piecewise constant function given by:
$$
H(x)=\left\{\begin{array}{ll}
1,&x> 0\\
\frac{1}{2},&x=0\\
0,&x<0
\end{array},
\right.
$$
and the Dirac delta function $\delta(\cdot)$~(more precisely, a distribution~(see, e.g.,~\cite{dirac})), a generalized function whose discrete analog is referred to as the Kronecker delta function:
$$\delta_{ij}=\left\{\begin{array}{ll}
1,&i=j\\
0,&i\neq j
\end{array}.
\right.
$$

This article concerns a newcomer, the shrinkage function $S_{\lambda}(\cdot)$, first introduced by Donoho and Johnstone in their landmark paper~(\cite{soft-threshold:donoho-johnstone}, see also \cite{soft-threshold2}) on function estimation using wavelets in the early 1990's. Recently, the shrinkage function has been heavily used in the solutions of several optimization and approximation problems of matrices (see, e.g.,~\cite{svt:cai-candes-shen,lin-chen-ma,tao-yuan,yuan-yang}). We give an elementary treatment that is accessible to a vast group of researchers, as it only requires basic knowledge in calculus and linear algebra and show how naturally the shrinkage function can be used in solving more advanced problems.

\section{A calculus problem} We start with a simple calculus problem. Let $\lambda>0$ and $a\in {\mathbb R}$ be given. Consider
the following problem:
\begin{equation}\label{calcp}
\min_{x\in {\mathbb R}}\left[\lambda|x|+\frac{1}{2}(x-a)^2\right].
\end{equation}
We adopt the notation $a=\displaystyle{\arg\min_{x\in A}f(x)}$ to mean that $a\in A$ is a solution of the minimization problem $\displaystyle{\min_{x\in A}f(x)}$ and define:
\begin{equation}\label{P}
S_{\lambda}(a):=\displaystyle{\arg\min_{x\in {\mathbb R}}\left[\lambda|x|+\frac{1}{2}(x-a)^2\right]}.
\end{equation}

\begin{theorem}\label{theorem 1}
{ Let $\lambda>0$ be fixed. For
each $a \in {\mathbb R}$, there is one and only one solution $S_{\lambda}(a)$, to
the minimization problem~(\ref{P}). Furthermore,
$$
S_{\lambda}(a)=\left\{\begin{array}{ll}
a-\lambda,&a>\lambda\\
0,&|a|\leq \lambda\\
a+\lambda,&a<-\lambda
\end{array}.\right.$$
}
\end{theorem}

\noindent {\bf Remark.} The function $S_{\lambda}(\cdot )$ defined
above is called the shrinkage function (also referred to as
soft shrinkage or soft threshold, \cite{soft-threshold:donoho-johnstone,soft-threshold2}). One may imagine that $S_{\lambda}(a)$ ``shrinks'' $a$ to zero when~$|a|\leq \lambda.$
A plot of $S_{\lambda}(\cdot)$  for $\lambda = 1$ is
given in Fig.~\ref{fig1}.
\begin{figure}[H]
	\centering
		\includegraphics[width=0.5\textwidth]{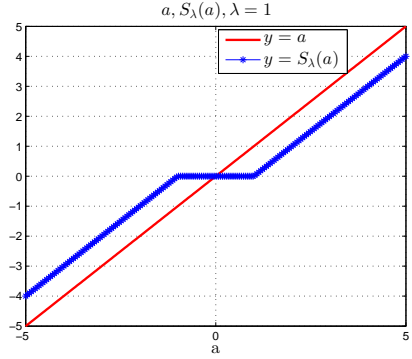}
		\caption{A plot of $S_{\lambda}$ for $\lambda=1$.}
		\label{fig1}
		\end{figure}

\begin{proof} Let $f(x)=\lambda|x|+\frac{1}{2}(x-a)^2$. Note that $f(x)\to\infty$ when $|x|\to\infty$ and $f$ is
continuous on ${\mathbb R}$ and differentiable everywhere except a single point $x=0$.
So, $f$ achieves its minimum value on ${\mathbb R}$ at one of its critical points.~A plot of $f$ for different values of $a$ and $\lambda=1$ is given in Fig.~\ref{fig2}.~Let $x^*=\displaystyle{\arg\min_{x\in {\mathbb R}}f(x)}$.
\begin{figure}[httb!]
	\centering
	\includegraphics[width=0.5\textwidth]{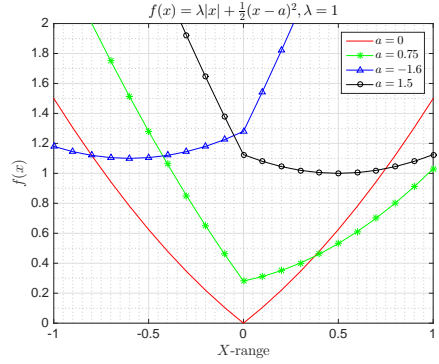}
	\caption{Plots of $f(x)$ for different values of $a$ with $\lambda=1$.}
	\label{fig2}
\end{figure}

We consider three cases.

Case 1: $x^*>0$. Since $f$ is differentiable at $x=x^*$ and achieves its minimum,
we must have $f'(x^*)=0$. Note that, for $x>0$, we have $$f'(x)=\frac{d}{dx}[\lambda x+\frac{1}{2}(x-a)^2]=\lambda+(x-a).$$
So, $$\lambda +(x^*-a)=0,$$
which implies $$x^*=a-\lambda.$$
To be consistent with $x^*>0$, it is necessary that $a-\lambda>0$ or, equivalently, $a>\lambda$.

Case 2: $x^*<0$. By proceeding similarly as in Case 1 above, we can arrive at
$$
x^*=a+\lambda ~~{\rm with }~~a<-\lambda.$$

Case 3: $x^*=0$. Note that $f(x)$ is no longer differentiable at $x^*=0$~(so we could not use the condition $f'(x^*)=0$ as before). But since $f$ has a minimum at $x^*=0$ and since $f$ is differentiable on each side of $x^*=0$, it is necessary that
$$f'(x)>0\;\;\text{for}\;\;{x>0}\;\;\text{and}\;\;f'(x)<0\;\text{for}\;\;x<0.$$
So,
$$ 
\lambda+x-a>0\;\;\text{for}\;x>0\;\;\text{and}\;-\lambda+x-a<0\;\;\text{for}\;x<0.
$$
Thus, 
$$ 
\lambda-a>0\;\;\text{and}\;-\lambda-a<0,
$$
or, equivalently,
$$
|a|\le\lambda.
$$
To summarize, we have
$$
x^*=\left\{
\begin{array}{ll}
a-\lambda &{\rm with}~a>\lambda,\\
a+\lambda &{\rm with}~a<-\lambda,\\
0 &{\rm with}~|a|\leq \lambda.
\end{array}
\right.$$

Since one and only one of the three cases (1) $a>\lambda$, (2)
$a<-\lambda$, and (3) $|a|\leq \lambda$ holds, we obtain the
uniqueness. With the uniqueness, it is straightforward
to verify that each of
the three cases would imply the corresponding formula for $x^*$.
\end{proof}

\section{A sparse recovery problem} Recently, research in
compressive sensing leads to the recognition of the fact that for many optimization problems, the $\ell_1$-norm
of a vector is a good substitute for the count of the number of
non-zero entries of the vector in many minimization problems. In this section, we solve some simple minimization problems using the count of non-zero entries or the $\ell_1$-norm.
Given a vector
${\mathbf v}\in {\mathbb R}^n$, we want to solve
\begin{equation}\label{SP}
\min_{{\mathbf u} \in \mathbb{R}^n}[\|{\mathbf u}\|_{\ell_0}+\frac{\beta}{2}\|{\mathbf u}-{\mathbf v}\|_{\ell_2}^2],
\end{equation}
where $\|\cdot\|_{\ell_0}$ denotes the number of non-zero entries of ${\mathbf u}$, such that, $\|{\mathbf u}\|_{\ell_0}=\#\{i|u_i\neq 0\},$ the cardinality of ${\mathbf u}$,~$\|\cdot\|_{\ell_2}$ denotes the
Euclidean norm in ${\mathbb R}^n$, and $\beta>0$ is a given balancing
parameter.~We can solve problem~(\ref{SP}) component-wise~(in each $u_i$) as follows.~Notice that, given $\mathbf{u} \in \mathbb{R}^n$, each entry $u_i$ of $\mathbf{u}$ contributes 1 to $\|{\mathbf u}\|_{\ell_0}$ if $u_i$ is non-zero, and contributes 0 if $u_i$ is zero.~Since we are minimizing $g({\mathbf u}):=\|{\mathbf u}\|_{\ell_0}+\frac{\beta}{2}\|{\mathbf u}-{\mathbf v}\|_{\ell_2}^2$, if $u_i$ is zero then the contribution to $g({\mathbf u})$ depending on this $u_i$ is $\frac{\beta}{2}v_i^2$; otherwise, if $u_i$ is non-zero, then we should minimize $\frac{\beta}{2}(u_i - v_i)^2$ for $u_i \in \mathbb{R}\setminus \{0\}$, which forces that $u_i = v_i$ and contributes 1 to $g({\mathbf u})$ as the minimum value.  Therefore, the solution $\mathbf{u}$ to problem (\ref{SP}) is given component-wise by
$$
u_i =
\begin{cases}
0, & \text{if }\frac{\beta}{2}(v_i)^2 \le 1 \\
v_i, & \text{ otherwise.}
\end{cases}
$$
Next, we replace $\|{\mathbf u}\|_{\ell_0}$ by $\|{\mathbf u}\|_{\ell_1}$ in (\ref{SP}) and solve:
\begin{equation}\label{P2}
\min_{{\mathbf u}\in \mathbb{R}^n}[\|{\mathbf u}\|_{\ell_1}+\frac{\beta}{2}\|{\mathbf u}-{\mathbf v}\|_{\ell_2}^2],
\end{equation}
where $\|\cdot \|_{\ell_1}$ denotes the $\ell_1$ norm in ${\mathbb
R}^n$.

Using Theorem~\ref{theorem 1}, we can solve (\ref{P2}) component-wise as follows.
\vspace{0.05in}
\begin{theorem}\label{theorem 2} \cite{yin-hale-zhang} { Let $\beta > 0$ and $\mathbf{v} \in \mathbb{R}^n$ be given and let $$ {\mathbf
u}^*=\arg\min_{{\mathbf u}\in \mathbb{R}^n}[\|{\mathbf
u}\|_{\ell_1}+\frac{\beta}{2}\|{\mathbf u}-{\mathbf v}\|_{\ell_2}^2].
$$
Then
$${\mathbf u}^*=S_{1/\beta}({\mathbf v}),$$
where, $S_{1/\beta}({\mathbf v})$ denotes the vector
whose entries are obtained by applying the shrinkage function
$S_{1/\beta}(\cdot)$ to the corresponding entries of ${\mathbf
v}$.}
\end{theorem}

\begin{proof}  If $u_i$ and $v_i$ denote the $i$th entry of the vectors $\mathbf{u}$ and $\mathbf{v}$, respectively, $i = 1,2,\dots,n$, then we have,
\begin{eqnarray*}
{\mathbf
u}^*&=&\arg\min_{{\mathbf u}\in \mathbb{R}^n}[\|{\mathbf
u}\|_{\ell_1}+\frac{\beta}{2}\|{\mathbf u}-{\mathbf v}\|_{\ell_2}^2] \\
&=& \arg\min_{\mathbf{u} \in \mathbb{R}^{n}} \left[\sum_{i=1}^n\left( \frac{1}{\beta}|u_i| + \frac{1}{2}(u_i -v_i)^2\right)\right].
\end{eqnarray*}
Since the $i$-th term in the summation depends  only on $u_i$, the vector $\mathbf{u}^*$ must have components $u_i^*$ satisfying
$$ u_i^* = \arg \min_{u_i^* \in \mathbb{R}}[\frac{1}{\beta}|u_i| + \frac{1}{2}(u_i - v_i)^2],$$
for $i = 1,2,\dots,n$.  But by Theorem~\ref{theorem 1}, the solution to each of these problems is given precisely by $S_{1/\beta}(v_i)$.  This yields the result.
\end{proof}
\\

\noindent
{\bf Remark.} The previous proof still works if we replace the vectors by
matrices and use the extension of the $\ell_1$ and $\ell_2$-norms to matrices {\it by treating them as vectors}.  Thus by using the same argument we can easily show the following matrix version of the previous theorem.
\vspace{0.05in}
\begin{theorem}\label{theorem 3}~\cite{yin-hale-zhang} 
Let $\beta > 0$ and $\mathbf{V} \in \mathbb{R}^{m\times n}$ be given.  Then
$$S_{1/\beta}(\mathbf{V}) = \arg \min_{\mathbf{U} \in \mathbb{R}^{m\times n}}[\|\mathbf{U}\|_{\ell_1} + \frac{\beta}{2}\|\mathbf{U} - \mathbf{V}\|^2_{\ell_2}],$$
where $S_{1/\beta}(\mathbf{V})$ is defined component-wise.
\end{theorem}\\
Theorem~\ref{theorem 3} solves the problem of approximating a given
matrix by a sparse matrix by using the shrinkage function.

\section{Approximation by low rank matrices}
The sparse approximation as given by Theorem~\ref{theorem 3} has many applications such as data compression and dimension reduction~\cite{siam,fazel}.~In these areas, one may also be interested in finding matrices of low rank~(see, e.g.,~\cite{lin-chen-ma,tao-yuan,yuan-yang,dutta_thesis}).~For example,
given a matrix ${\mathbf A}\in {\mathbb R}^{m\times n}$, we want
to solve the following approximation problem:
\begin{equation*}
\min_{{\mathbf X \in \mathbb{R}^{m\times n}}}\left[{\rm rank}({\mathbf X})+\frac{\beta}{2}\|{\mathbf X}-{\mathbf A}\|_F^2\right],
\end{equation*}
where $\|\cdot\|_F$ denotes the Frobenius norm of matrices (which
turns out to be equivalent to the vector $l_2$ norm if we treat a matrix as a
vector - see more discussion on the matrix norms in subsection 4.1).

This is a harder problem since rank$(X)$ is not a convex function.~A convex relaxation (see, e.g., ~\cite{fazel}) of
the problem is provided by replacing the term ${\rm rank}({\mathbf
X})$ by the nuclear norm of ${\mathbf X}$, $\|{\mathbf X}\|_*$,~(again, see subsection 4.1 for a discussion on the nuclear norm and its properties).~The problem then becomes:
\begin{equation}\label{NN}
\min_{{\mathbf X}\in \mathbb{R}^{m\times n}}\left[\|{\mathbf X}\|_*+\frac{\beta}{2}\|{\mathbf X}-{\mathbf A}\|_F^2\right].
\end{equation}
This problem again yields an explicit solution
(\cite{svt:cai-candes-shen,svt2}), but in these literatures, the formula is derived by using advanced tools from convex analysis
(``subdifferentials'' to be more specific)~\cite{watson}. Here,
we will show how we can obtain the solution
by using simple ideas from the previous section.
\subsection{Singular value decomposition and matrix norms} It will be beneficial to
recall the various matrix norms. Many useful matrix norms can be defined in terms of the singular
values of the matrices. We will deal with two of them: the nuclear norm $\|\cdot \|_{*}$ and the Frobenius norm $\|\cdot \|_F$.

Let ${\mathbf A}\in{\mathbb R}^{m\times n}$ and ${\mathbf A}={\mathbf U}\tilde{\mathbf A}{\mathbf V}^{T}$ be a singular value decomposition~(SVD)
of ${\mathbf A}$ with ${\mathbf U}\in {\mathbb R}^{m\times m}$ and ${\mathbf V}\in {\mathbb R}^{n\times n}$ being two
orthogonal matrices (that is, ${\mathbf U}^{-1}={\mathbf U}^{T}$ and ${\mathbf V}^{-1}={\mathbf V}^{T}$)
and $\tilde{\mathbf A}={\rm diag}(\sigma_1({\mathbf A}),\;\sigma_2({\mathbf A}),\ldots, \sigma_{\min \{m,n\}}({\mathbf A}))$ being a $m\times n$ non-square diagonal matrix
having $\sigma_1({\mathbf A})\geq \sigma_2({\mathbf A})\geq \cdots \geq \sigma_{\min \{m,n\}}({\mathbf A})\geq 0$ on its diagonal and 0's elsewhere. The $\sigma_i({\mathbf A})$'s are
called the singular values of ${\mathbf A}$.
It is known that every matrix in ${\mathbb R}^{m\times n}$ has a SVD and that SVD of a matrix is not necessarily unique~\cite{linearAlgebra}. 
Then the nuclear norm of ${\mathbf A}$ is defined as
$$
\|{\mathbf A}\|_*=\sum_{i=1}^{\min \{m,n\}}\sigma_i({\mathbf A}),$$
and the Frobenius norm  of ${\mathbf A}$ as
$$
\|{\mathbf A}\|_F=\left(\sum_{i=1}^{\min \{m,n\}}(\sigma_i({\mathbf A}))^2\right)^{1/2}.$$
This norm turns out to be the same as the $\ell_2$ norm of ${\mathbf A}$, treated as a vector in $\mathbb{R}^{mn\times 1}$.
This is because the nonzero singular values $\sigma_i({\mathbf A})$'s are exactly the square root of the nonzero eigenvalues of ${\mathbf A}{\mathbf A}^T$ or ${\mathbf A}^T{\mathbf A}$. 
So, $$
\|{\mathbf A}\|_{l_2}^2=\sum_{i=1}^m\sum_{j=1}^n(a_{ij})^2={\rm trace}({\mathbf A}{\mathbf A}^{T})
=\sum_{i=1}^{\min \{m,n\}}(\sigma_i({\mathbf A}))^2.$$
Here we have used ${\rm trace}(\cdot)$ to denote the trace of a matrix (which is equal to the sum of all
diagonal entries of the matrix).

We will need the following simple and well-known fact about the nuclear norms of a matrix and that of its diagonal~\cite{stewart-sun}: Let $D({\mathbf A})$ denote the diagonal matrix using the diagonal of ${\mathbf A}$ then:
\begin{equation}\label{diag}
\|D({\mathbf A})\|_*\leq \|{\mathbf A}\|_*.
\end{equation}
This inequality can be verified by using a SVD of ${\mathbf A}={\mathbf U}\tilde{\mathbf A}{\mathbf V}^{T}$ as follows.~For completeness we will provide a simple proof of~(\ref{diag}).~Write ${\mathbf U}=(u_{ij})$, ${\mathbf V}=(v_{ij})$, and $t=\min\{m,n\}$. Then
$$
\|D({\mathbf A})\|_*=\|D({\mathbf U}\tilde{{\mathbf A}}{\mathbf V}^{T})\|_*
=\sum_{i=1}^{m}\left|\sum_{j=1}^{t}\sigma_j({\mathbf A})u_{ij}v_{ij}\right|
\leq \sum_{j=1}^{t}\sigma_j({\mathbf A})\sum_{i=1}^m|u_{ij}v_{ij}|
~~~$$
$$
\leq \sum_{j=1}^{t}\sigma_j({\mathbf A})\cdot \left(\sum_{i=1}^{m}|u_{ij}|^2\right)^{1/2}\left(\sum_{i=1}^{m}|v_{ij}|^2\right)^{1/2}
\leq
\sum_{j=1}^{t}\sigma_j({\mathbf A})=\|{\mathbf A}\|_*,
$$
where we have used the Cauchy-Schwarz inequality in obtaining the second inequality, and the orthogonality of ${\mathbf U}$ and ${\mathbf V}$ (so that $\sum_{i=1}^m|u_{ij}|^2\leq 1$ and $\sum_{i=1}^m|v_{ij}|^2\leq 1$) in the last inequality.

We will also use the fact that for any orthogonal matrices ${\mathbf L}$ and ${\mathbf R}$, $\mathbf{LAR}$ and ${\mathbf A}$ have the same singular values, and therefore their Frobenius norms and nuclear norms are same:
$$\|\mathbf{LAR}\|_F=\|\mathbf{A}\|_F\;\text{and}\;\|\mathbf{LAR}\|_*=\|\mathbf{A}\|_*.$$
This is known as the unitary invariance of the Frobenius norm and nuclear norm. 

\subsection{Solution to (\ref{NN}) via problem (\ref{P})} We are ready to show how problem (\ref{NN}) is problem (\ref{P}) in disguise.~Let ${\mathbf U}\tilde{\mathbf A}{\mathbf V}^{T}$ be a SVD of ${\mathbf A}$.~Given $\beta> 0$, using the unitary invariance of the Frobenius norm and the nuclear norm, we have
\begin{eqnarray*}
\min_{{\mathbf X}}[ \|{\mathbf X}\|_*+\frac{\beta}{2}\|{\mathbf X}-{\mathbf A}\|_F^2]&=&\min_{{\mathbf X}}[\|{\mathbf X}\|_*+\frac{\beta}{2}\|{\mathbf X}-{\mathbf U}\tilde{\mathbf A}{\mathbf V}^{T} \|_F^2]\\
&=&\min_{{\mathbf X}}[\|{\mathbf U}^{T}{\mathbf X}{\mathbf V}\|_*+\frac{\beta}{2}\|{\mathbf U}^{T}{\mathbf X}{\mathbf V}-\tilde{\mathbf A}\|_F^2].
\end{eqnarray*}
It can be seen from the last expression that the minimum occurs when
$\tilde{\mathbf X}:={\mathbf U}^{T}{\mathbf X}{\mathbf V}$ is diagonal since both terms in that expression get no larger when
$\tilde{\mathbf X}$ is replaced by its diagonal matrix (with the help of (\ref{diag})). So, the matrix
$
{\mathbf E}=({e}_{ij}):=\tilde{\mathbf X}-\tilde{\mathbf A}
$
is a diagonal matrix. Thus, $${\mathbf X}={\mathbf U}\tilde{\mathbf X}{\mathbf V}^{T},
 ~{\rm with}~\tilde{\mathbf X}=\tilde{\mathbf A}+{\mathbf E},$$
which yields a SVD of ${\mathbf X}$ (using the same matrices ${\mathbf U}$ and ${\mathbf V}$ as in a SVD of ${\mathbf A}$~!). Then,
\begin{eqnarray}\label{svt_1}
\min_{{\mathbf X}}[\|{\mathbf X}\|_*+\frac{\beta}{2}\|{\mathbf X}-{\mathbf A}\|_F^2]&=&
\min_{\tilde{\mathbf X}\in {\rm diag}}[\|\tilde{\mathbf X}\|_*+\frac{\beta}{2}\|\tilde{\mathbf X}-\tilde{\mathbf A}\|_F^2]\nonumber\\
&=&\min_{\tilde{\mathbf X}\in {\rm diag}}[ \sum_{i}|\tilde{x}_{ii}|+\frac{\beta}{2}\sum_{i}(\tilde{x}_{ii}-\sigma_i({\mathbf A}))^2],
\end{eqnarray}
where ${\rm ``diag"}$ is the set of diagonal matrices in $\mathbb{R}^{m\times n}$.~Note that (\ref{svt_1}) is an optimization problem like (\ref{P}) (for vectors $(\tilde{x}_{11},\tilde{x}_{22},\cdots)^T$
as $\tilde{\mathbf X}$ varies) whose solution is given by
$$
\tilde{x}_{ii}=S_{{1}/{\beta}}(\sigma_i({\mathbf A})),~i=1,2,\cdots.$$
To summarize, we have proven the following.
\vspace{0.05in}
\begin{theorem} \label{theorem 4}\cite{svt:cai-candes-shen} {Suppose that ${\mathbf A} \in \mathbb{R}^{m\times n}$ and $\beta > 0$ are given.
Then the solution to the minimization problem (\ref{NN})
is given by $\hat{\mathbf X}={\mathbf U}\tilde{\mathbf X}{\mathbf V}^{T}$ where the diagonal matrix $\tilde{\mathbf{X}}$ has diagonal entries
$$\tilde{x}_{ii} = S_{{1}/{\beta}}(\sigma_i({\mathbf{A}})), ~i=1,2,\dots,\min\{m,n\}, $$
where ${\mathbf U}\tilde{\mathbf A}{\mathbf V}^{T}$ is a SVD of ${\mathbf A}$.
}
\end{theorem}
\vspace{0.05in}

\noindent {\bf Remark.}
{\bf 1.} A recent proof of this
theorem is given by Cai, Candes, and Shen in
\cite{svt:cai-candes-shen} where they give an advanced
verification of the result.  Our proof given above has the advantage that it is elementary and allows the reader to ``discover'' the result.\\
{\bf 2.} There are many earlier discoveries of related results
(\cite{linearAlgebra, strang}) where rank$({\mathbf X})$ is used instead of the
nuclear norm $\|{\mathbf X}\|_*$. We will examine one
such variant in the next section.\\
{\bf 3.} A closely related (but harder) problem is {\it compressive sensing} (\cite{candes-romberg-tao,candes-tao2010}).~In a more general set up where only a subset of the entries of the data matrix is observable, for example, matrix completion problem under low-rank penalties~\cite{svt:cai-candes-shen,svt2,candes-tao2010}:
\begin{equation*}
\min_{{\mathbf X}}{\rm rank}({\mathbf X})~~~~{\rm subjcet~to~} A_{ij}=X_{ij},~(i,j)\in\Omega,
\end{equation*}
where $\Omega\subseteq\{(i,j):1\le i\le m, 1\le j\le n\}$,~is indeed NP-hard~\cite{NP}.~Note that, the matrix completion problem is a special case of the affinely constrained matrix rank minimization problem~\cite{svt2}: \begin{equation*}
\min_{{\mathbf X}}{\rm rank}({\mathbf X})~~~~{\rm subjcet~to~} A({\mathbf X})={\mathbf b},
\end{equation*} where ${\mathbf X}\in\mathbb{R}^{m\times n}$ be the decision variable and~${\mathbf A}:\mathbb{R}^{m\times n}\to\mathbb{R}^p$ be a linear map. One common idea used in such a situation is to consider a convex relaxation of the above problem by replacing the rank by its convex surrogate, the nuclear norm.~Readers are strongly recommended to the recently survey by Bryan and Leise (\cite{siam}).

\section{A variation} More problems can be solved by applying similar ideas. For example, let us consider a
variant of a well-known result of Schmidt~(see, e.g.,~\cite[Section 5]{strang}), replacing the rank by the nuclear
norm: For a fixed positive number $\tau$, consider
\begin{equation}\label{BA}
\min_{\mathbf X\in \mathbb{R}^{m\times n}} \|{\mathbf X}-{\mathbf A}\|_F
~~{\rm subject~to}~~\|{\mathbf X}\|_*\leq \tau.
\end{equation}
Using similar methods as in Section 3, this problem can be transformed~(see the derivation below) into the following:
\begin{equation}\label{BAS}
\min_{\mathbf u\in \mathbb{R}^{\min\{m,n\}}} \|{\mathbf u}-{\mathbf v}\|_{\ell_2}
~~{\rm subject~to}~~\|{\mathbf u}\|_{\ell_1}\leq \tau.
\end{equation}

The LASSO~(Least absolute shrinkage and selection operator) is a vastly used regression technique in data mining and statistics~\cite{soft-threshold2,Gen_lasso,LASSO_dual}.~It follows a simple model for variable selection and regularization.~Let the predictor variables, ${\mathbf X}\in\mathbb{R}^{N \times p}$ and the responses, $y_i, i = 1,2,\ldots, N$ are given.~Assuming that~$\sum_ix_{ij}/N=0$ and $\sum_ix_{ij}^2/N=1$, the LASSO estimate is given by:
\begin{equation}\label{lasso_problem}
{\mathbf {\hat{\beta}}}=\arg\min_{{\mathbf \beta}}\|{\mathbf y}-{\mathbf {\beta_0}}-{\mathbf {X\beta}}\|_{\ell_2}^2,~{\rm subject~to}~\|\mathbf{\beta}\|_{\ell_1}\le\tau,
\end{equation}
where ${\mathbf {\beta_0}}$ is the mean of the response vector ${\mathbf y}$.~Note that,~(\ref{BAS}) shares some similarity with~(\ref{lasso_problem})~\cite{soft-threshold2,Gen_lasso,LASSO_dual}.~As in~\cite{soft-threshold2}, one can form a {\it Lagrangian} of~(\ref{BAS}) and solve:
\begin{equation}\label{Lagrange BAS1}
{\mathbf u^*}=\arg\min_{\mathbf u\in \mathbb{R}^{\min\{m,n\}}} \{\frac{1}{2}\|{\mathbf u}-{\mathbf v}\|_{\ell_2}^2+\lambda\|{\mathbf u}\|_{\ell_1}\},\;\;{\rm with}\;\|S_{\lambda}({\mathbf v})\|_{\ell_1}=\tau,
\end{equation}
which has a solution ${\mathbf u^*}=S_{\lambda}({\mathbf v})$ according to Theorem~\ref{theorem 2}.~(The reason for us to use $\lambda$ instead of $\beta$ in~(\ref{BAS}) is nonessential: it is only for indicating the similarity with LASSO formulation.)
We now sketch the derivation of converting~(\ref{BA}) to~(\ref{BAS}):
As before, let ${\mathbf A} = {\mathbf U}\tilde{\mathbf A}{\mathbf V}^{T}$ be a SVD of ${\mathbf A}$. Then,
$$\min_{\mathbf X\in \mathbb{R}^{m\times n}}\|\mathbf{X - A}\|_F = \min_{\mathbf{X}\in \mathbb{R}^{m\times n}}\|\mathbf{U^{T}XV} - \tilde{\mathbf A} \|_F.$$ Note that $\|{\mathbf X}\|_*=\|{\mathbf U}^{T}{\mathbf X}{\mathbf V}\|_*$, so (\ref{BA}) can be written as
\begin{equation*}
\min_{\mathbf X\in \mathbb{R}^{m\times n}} \|{\mathbf X}-\tilde{\mathbf A}\|_F
~~{\rm subject~to}~~\|{\mathbf X}\|_*\leq \tau,
\end{equation*}
which, by using (\ref{diag}), can be further transformed to
\begin{equation}\label{BAd}
\min_{\mathbf X\in \mathbb{R}^{m\times n}} \|{\mathbf X}- \tilde{\mathbf A}\|_F
~~{\rm subject~to}~~{\mathbf X} ~{\rm being}~{\rm diagonal}~ {\rm and}~\|{\mathbf X}\|_*\leq \tau.
\end{equation}
Next, if we let $\mathbf{u}$ and $\mathbf{v}$ be two vectors in $\mathbb{R}^{\min\{m, n\}}$ consisting of the diagonal elements of $\mathbf{{X}}$ and ${\tilde{\mathbf A}}$, respectively, then
(\ref{BAd}) is equivalent to~(\ref{BAS}). 
Since $S_{\lambda_1}({\mathbf v})$ solves (\ref{Lagrange BAS1}) we have, 
\begin{align*}
\frac{1}{2}\|S_{\lambda_1}({\mathbf v})-{\mathbf v}\|_2^2+\lambda_1\tau \le \frac{1}{2}\|{\mathbf u}-{\mathbf v}\|_{\ell_2}^2+\lambda_1\|{\mathbf u}\|_{\ell_1},
\end{align*}
for all ${\mathbf u\in \mathbb{R}^{\min\{m,n\}}}$. Which implies, for all ${\mathbf u\in \mathbb{R}^{\min\{m,n\}}}$, 
\begin{align*}
\frac{1}{2}\|S_{\lambda_1}({\mathbf v})-{\mathbf v}\|_2^2 \le \frac{1}{2}\|{\mathbf u}-{\mathbf v}\|_{\ell_2}^2+\lambda_1(\|{\mathbf u}\|_{\ell_1}-\tau).
\end{align*}
Therefore, 
\begin{align*}
\frac{1}{2}\|S_{\lambda_1}({\mathbf v})-{\mathbf v}\|_2^2 \le \{\frac{1}{2}\|{\mathbf u}-{\mathbf v}\|_{\ell_2}^2\},
\end{align*}
for all ${\mathbf u\in \mathbb{R}^{\min\{m,n\}}}$, such that $\|{\mathbf u}\|_{\ell_1}\le\tau$. Hence ${\mathbf u^*}=S_{\lambda}({\mathbf v})$ solves~(\ref{BAS}).
Thus we have established the following result.
\vspace{0.05in}
\begin{theorem}\label{theorem 5}\cite{dutta_thesis} With the notations above,  the solution to problem (\ref{BA}) is
	given by
	$$
	\hat{\mathbf X} ={\mathbf  U} S_{\lambda}(\tilde{\mathbf A}) {\mathbf V}^{T},
	$$
	for some $\lambda$ such that $\|S_{\lambda}(\tilde{\mathbf A})\|_{\ell_1}=\tau.$
\end{theorem}\\
\\
\noindent
{\bf Acknowledgment.}
This work is partially supported
by a CSUMS program of the National Science Foundation DMS-0803059. Toby Boas, Katie Mercier, and Eric Niederman contributed to this work as undergraduate students.

\end{document}